\documentclass[10pt,a4paper]{article}
\usepackage{amsmath}
\usepackage{amssymb}
\usepackage{amsthm}
\usepackage{amsfonts}
\usepackage{enumerate}
\usepackage{url,hyperref}
\usepackage{verbatim}
\usepackage{mathtools} 
\usepackage{multirow}
\usepackage{algorithm,algorithmic}
\usepackage[margin=1.3in]{geometry}
\def\tluste#1{\protect{\textrm{\boldmath $#1$}}}
\newcommand{\vr}[1]{{{#1}}}

\newcommand{\mace}[1]{{{#1}}}
\newcommand{\mna}[1]{{\mathcal{#1}}}

\newcommand{\omace}[1]{\mbox{$\overline{\mace{#1}}$}} 
\newcommand{\umace}[1]{\mbox{$\underline{\mace{#1}}$}} 
\newcommand{\imace}[1]{\tluste{#1}} 		

\def\Mid#1{#1_c}

\newcommand{\ovr}[1]{\mbox{$\overline{\vr{#1}}$}} 
\newcommand{\uvr}[1]{\mbox{$\underline{\vr{#1}}$}}
\newcommand{\onum}[1]{\mbox{$\overline{{#1}}$}} 
\newcommand{\unum}[1]{\mbox{$\underline{{#1}}$}}

\newcommand{\inum}[1]{\mbox{$\tluste{#1}$}} 

\newcommand{\R}[0]{{\mathbb{R}}}

\def\eps{{\varepsilon}}

\newcommand{\mmid}[0]{;\,}		

\newcommand{\seznam}[1]{{\{1, \ldots, {#1}\}}}

\def\clqq{``}
\def\crqq{''}
\def\quo#1{\clqq{}#1\crqq{}}  

\DeclareMathOperator{\rank}{rank}

\def\xxx#1{\fbox{\sc #1}}

\newtheorem{theorem}{Theorem}

\newtheorem{proposition}{Proposition}

\theoremstyle{definition}

\newtheorem{example}{Example}
\newtheorem{remark}{Remark}

\begin{document}

\title{Eigenvalues of symmetric tridiagonal interval matrices revisited}

\author{
  Milan Hlad\'{i}k\footnote{
Charles University, Faculty  of  Mathematics  and  Physics,
Department of Applied Mathematics, 
Malostransk\'e n\'am.~25, 11800, Prague, Czech Republic, 
e-mail: \texttt{milan.hladik@matfyz.cz}
}
}

\date{July 9, 2018}
\maketitle

\begin{abstract}
In this short note, we present a novel method for computing exact lower and upper bounds of eigenvalues of a symmetric tridiagonal interval matrix. Compared to the known methods, our approach is fast, simple to present and to implement, and avoids any assumptions. Our construction explicitly yields those matrices for which particular lower and upper bounds are attained.
\end{abstract}


\section{Introduction}

Consider a symmetric tridiagonal matrix of size $n$
\begin{align*}
A=\begin{pmatrix}
 a_1  & b_2  &  0  & \dots  & 0 \\
 b_2  & a_2  & b_3 & \ddots & \vdots \\
 0    & b_3  & a_3 & \ddots & 0 \\
\vdots&\ddots&\ddots& \ddots & b_n \\
 0    &\dots & 0   & b_n & a_n
\end{pmatrix}.
\end{align*}
and denote by $\lambda_1(A)\geq\dots\geq\lambda_n(A)$ its eigenvalues. Assume that the entries of $A$ are not known precisely and the only information that we have is that $a_i$ comes from a given interval $\inum{a}_i=[\unum{a}_i,\onum{a}_i]$, $i=1,\dots,n$, and $b_i$ comes from a given interval $\inum{b}_i=[\unum{b}_i,\onum{b}_i]$, $i=2,\dots,n$.

By $\imace{A}=[\umace{A},\omace{A}]$ we denote the corresponding interval matrix, that is, the set of all matrices with $a_i\in\inum{a}_i$ and  $b_i\in\inum{b}_i$.
By $\Mid{A}\coloneqq\frac{1}{2}(\umace{A}+\omace{A})$ we denote its midpoint. Next, 
\begin{align*}
\inum{\lambda}_i=[\unum{\lambda}_i,\onum{\lambda}_i]
\coloneqq\{\lambda_i(A)\mmid A\in\imace{A}\}
\end{align*}
stands for the corresponding eigenvalue sets. It was shown in \cite{HlaDan2010} that they form real compact intervals. The problem investigated in this paper is to determine their end-points. We focus on the upper end-points $\onum{\lambda}_i$s since the lower ones can be determined analogously by the reduction $A\mapsto -A$.

Characterization of the extremal eigenvalues  $\onum{\lambda}_1$ and $\unum{\lambda}_n$ of a general symmetric interval matrix is due to Hertz \cite{Her1992} by a formula involving computation of $2^n$ matrices. A partial characterization of the intermediate eigenvalue intervals was done in \cite{Dei1991a,HlaDan2011c}. Due to NP-hardness of computing or even tightly approximating the eigenvalue sets \cite{Hla2015b,Roh1994}, there were developed various outer and inner approximation methods \cite{AhnMoor2006,HlaDan2010,HlaDan2011c,HlaDan2011b,Kol2010,Len2014,QiuChe1996,Roh1996c}. The tridiagonal case was particularly investigated by Commer\c{c}on \cite{Com1994}, who proposed a method for calculating the exact eigenvalue bounds based on the Sturm algorithm. This method, however, suffers from time complexity analysis and relies too much on the particular Sturm algorithm. Our aim is to have a finite reduction to real cases, which can be solved by any eigenvalue method for tridiagonal matrices. Another author investigating tridiagonal interval matrices was Jian \cite{Jia2017}.
He proposed a method for computing the extremal eigenvalues by a reduction to four real cases, and he also inspected tridiagonal interval Toeplitz matrices. Our approach Generalizes the result and enables to calculate ranges of all eigenvalue sets under eigenvector sign invariancy condition.

\section{Preliminaries}\label{sPre}

Throughout this paper, inequalities such as \quo{$\geq$} are applied entrywise. In particular, $A\geq0$ means that $A$ is entrywise nonnegative.

\begin{proposition}\label{propPreNonneg}
Without loss of generality, we can assume that ${A}\geq0$.
\end{proposition}

\begin{proof}
The transformation $A\mapsto A+\alpha I_n$ increases all eigenvalues of $A$ by the amount of $\alpha$. So for any $\alpha\geq\max_i\{-a_i\}$ this transformation yields a matrix with a nonnegative diagonal. Thus, we can assume that $a_i\geq0$ for every $i$.

Suppose now there is $i$ such that $b_i<0$. Let $\lambda$ be any eigenvalue of $A$ and $x$ a corresponding eigenvector, that is, $Ax=\lambda x$. Let $A'$ be the matrix resulting from $A$ by putting $b'_i=-b_i$, and let $x'=(-x_1,\dots,-x_{i-1},x_i,\dots,x_n)^T$. Then for $k>i$ we have
$$
(A'x')_k=(Ax)_k=(\lambda x)_k=(\lambda x')_k.
$$
For $k<i-1$ we have
$$
(A'x')_k=(-Ax)_k=(-\lambda x)_k=(\lambda x')_k.
$$
The remaining two cases are:
\begin{align*}
(A'x')_{i-1}
=b'_{i-1}x'_{i-2}+a'_{i-1}x'_{i-1}+b'_ix'_i
=-b_{i-1}x_{i-2}-a_{i-1}x_{i-1}-b_ix_i
=-\lambda x_{i-1}
=\lambda x'_{i-1},
\end{align*}
and
\begin{align*}
(A'x')_{i}
=b'_{i}x'_{i-1}+a'_{i}x'_{i}+b'_{i+1}x'_{i+1}
=b_{i}x_{i-1}+a_{i}x_{i}+b_{i+1}x_{i+1}
=\lambda x_{i}
=\lambda x'_{i}.
\end{align*}
Thus, $A'$ has the same eigenvalues as $A$, and the eigenvectors of $A$ can easily be derived from those of $A'$. By repeating this process, we obtain all $b_i$s nonnegative.
\end{proof}


We can therefore assume that $\umace{A}\geq0$ for the interval matrix $\imace{A}$. Nonnegativity of the diagonal can be achieved by the transformation $\inum{a}_i\mapsto \inum{a}_i+\alpha$ with $\alpha\coloneqq\max_i\{-\unum{a}_i\}$, and nonnegativity of the remaining entries by the transformation
\begin{align*}
\inum{b}_i\mapsto 
\begin{cases}
\inum{b}_i, & \mbox{if } \unum{b}_i\geq0,\\
-\inum{b}_i, & \mbox{if } \onum{b}_i\leq0,\\
[0,\max\{-\unum{b}_i,\onum{b}_i\}], & \mbox{otherwise.}
\end{cases}
\end{align*}


We will assume throughout the paper that $\onum{b}_i>0$ for all $i=2,\dots,n$; otherwise $\imace{A}$ is block diagonal and we split the problem into the subproblems corresponding to the diagonal blocks of $\imace{A}$.

\begin{proposition}\label{propSimple}
Suppose that $\unum{b}_i>0$ for all $i=2,\dots,n$.
Then all eigenvalues of every $A\in\imace{A}$ are simple.
\end{proposition}

\begin{proof}
It is obvious since it is known that a symmetric tridiagonal matrix has simple eigenvalues provided off-diagonal elements are nonzero \cite{Par1998}.
\end{proof}

\section{Sign invariancy case}

We say that eigenvectors of $\imace{A}$ are \emph{sign invariant} \cite{DeiRoh1994,Roh1993b} if to each eigenvalue $\lambda_i(A)$ we can associate an eigenvector $x^i(A)$ such that the signs of the entries of $x^i(A)$ are constant for $A\in\imace{A}$.
In this section, we assume that sign invariancy is satisfied.

The derivative of a simple eigenvalue $\lambda$ of a symmetric $A$ with respect to $a_{ij}$ is equal to $x_ix_j$, where $x$, $\|x\|_2=1$, is the corresponding eigenvector. The derivative is nonnegative with respect to the diagonal entries of $A$, so the largest eigenvalues of $\imace{A}$ are attained for $a_i\coloneqq\ovr{a}_i$. 
Notice that similar result holds for general symmetric interval matrices, too \cite{HlaDan2010,Kol2007}.

Due to sign invariancy of eigenvectors, we can easily determine also $b_i$-s. Let $\lambda_k$ be the $k$th largest eigenvalue of $\Mid{A}$ and $x$ the corresponding eigenvector. 
If $x_ix_{i+1}>0$, then $\onum{\lambda}_k$ is attained for $b_i=\onum{b}_i$. Otherwise, it is attained for $b_i=\unum{b}_i$.
In particular, from the Perron theory and properties of nonnegative matrices, we have that $\onum{\lambda}_1$ is attained for $A\coloneqq\omace{A}$.
The following proposition summarizes the result.

\begin{proposition}\label{propAttain}
$\onum{\lambda}_k$ is attained for
\begin{align*}
a_i\coloneqq\ovr{a}_i,\quad
b_i\coloneqq\begin{cases}
\onum{b}_i & \mbox{if }x_ix_{i+1}>0,\\
\unum{b}_i & \mbox{otherwise,}
\end{cases}
\end{align*}
where $x$ is the eigenvector of $\Mid{A}$ corresponding to $\lambda_k(\Mid{A})$.
\end{proposition}

\begin{remark}\label{remInner}
Notice that provided the problem is not sign invariant, then the eigenvalues computed by Proposition~\ref{propAttain} give an inner estimation of the eigenvalue intervals $\inum{\lambda}_1,\dots,\inum{\lambda}_n$. That is, we have intervals $\inum{\mu}_1,\dots,\inum{\mu}_n$ satisfying $\inum{\mu}_i\subseteq\inum{\lambda}_i$ for every $i=1,\dots,n$, with equality under sign invariancy.
\end{remark}

The resulting method is displayed in Algorithm~\ref{algGen} for computing the right end-points of the eigenvalue intervals; the left end-point are computed analogously.

\begin{algorithm}[t]
\caption{Computation of $\onum{\lambda}_1,\dots,\onum{\lambda}_n$ under sign invariancy.\label{algGen}}
\begin{algorithmic}[1]
\ENSURE{$\umace{A}\geq0$ }
\STATE
compute the eigenvalues $\lambda_1,\dots,\lambda_n$ and the corresponding eigenvectors $x^1,\dots,x^n$ of $\Mid{A}$
\STATE
for $i=1,\dots,n$ put $a_i\coloneqq\ovr{a}_i$ 
\FOR{$k=1,\ldots,n$}
\STATE
for $i=2,\dots,n$ put $b_i\coloneqq\ovr{b}_i$ if $x^k_ix^k_{i+1}>0$ and $b_i\coloneqq\uvr{b}_i$ otherwise
\STATE
compute the $k$th eigenvalue $\lambda_k(A)$ of the tridiagonal matrix $A$ with entries $a,b$
\STATE
put $\onum{\lambda}_k\coloneqq\lambda_k(A)$
\ENDFOR
\RETURN $\onum{\lambda}_1,\dots,\onum{\lambda}_n$
\end{algorithmic}
\end{algorithm}

As a side effect, we have the following interesting property.

\begin{proposition}\label{propCard}
$\onum{\lambda}_k$ is attained for $a_i\coloneqq\ovr{a}_i$ and $b_i\in\{\unum{b}_i,\onum{b}_i\}$ such that the cardinality of 
$$
\{i=2,\dots,n\mmid b_i=\onum{b}_i\}
$$
is $n-k$.
\end{proposition}

\begin{proof}
Suppose $\uvr{b}>0$; the general case then follows from limit transition due to continuity of eigenvalues.
Let $A\in\imace{A}$, let $\lambda_k$ be its $k$th eigenvalue and $x$ a corresponding eigenvector. By \cite[Thm.~7.9.2]{Par1998}, the sign of $x_j$ is equal to the sign of
$$
\chi_{j-1}(\lambda_k)b_{j+1}\dots b_n,
$$
where $\chi_{j-1}$ is the characteristic polynomial of the (top left) principal leading submatrix of $A$ of size $j-1$, and $\chi_{0}\equiv 1$. Since $b>0$, the signs of $x_j$ and $\chi_{j-1}(\lambda_k)$ coincide. The number of sign agreements between consecutive terms in the Sturm sequence $\{\chi_i(\lambda_k)\mmid i=0,1,\dots,n\}$ gives the number of roots of $\chi_n$ which are less than $\lambda_k$, that is $n-k$. Sign agreement between consecutive terms in the Sturm sequence corresponds to sign agreement between consecutive terms in the signs of the eigenvector $x$, which in turn sets $b_i$ to be $\onum{b}_i$ by Proposition~\ref{propAttain}.
Therefore, by the analysis of our method, $n-k$ is equal to the number of $b_i$, $i=2,\dots,n$, that we set to the right end-point.
\end{proof}



Time complexity of our algorithm is the following. We need computation of eigenvalues and eigenvectors of the midpoint matrix $\Mid{A}$, then $2n$-times computation of a certain eigenvalue of a matrix in $\imace{A}$. The preprocessing carrying the matrix $\imace{A}$ to the nonnegative form (Section~\ref{sPre}) requires only linear time. Provided we employ a standard method for computation of eigenvalues of a real symmetric tridiagonal matrix running in $\mna{O}(n^2)$, the overall complexity is $\mna{O}(n^3)$.

\section{General case}

As a simple corollary of Proposition~\ref{propCard} we get that $\onum{\lambda}_1$ is attained for $b\coloneqq\ovr{b}$ and $\onum{\lambda}_n$ is attained for $b\coloneqq\uvr{b}$. This property, however, holds in the general case and no sign invariancy assumption is needed. By other means, this was observed by Jian \cite{Jia2017}.

\begin{proposition}
$\onum{\lambda}_1$ and $\unum{\lambda}_n$ are attained for $b\coloneqq\ovr{b}$, and $\unum{\lambda}_1$ and $\onum{\lambda}_n$ are attained for $b\coloneqq\uvr{b}$.
\end{proposition}

\begin{proof}
Since $A\geq0$ for every $A\in\imace{A}$, the eigenvector associated to $\lambda_1(A)$ is the nonnegative Perron vector. Therefore, the eigenvector is always sign invariant and Proposition~\ref{propCard} applies. Similarly for $\lambda_n(A)$ by the transformation $A\mapsto-A$.
\end{proof}

Each of the quantities $\onum{\lambda}_1$, $\unum{\lambda}_1$, $\onum{\lambda}_n$ and $\unum{\lambda}_n$ are computable just by solving one real eigenvalue problem. 
As a consequence, we have a $\mna{O}(n^2)$ method for testing the following properties of a symmetric tridiagonal interval matrix $\imace{A}$, because they reduce to computation of eigenvalues of one or two real symmetric tridiagonal matrices:
\begin{itemize}
\item
positive (semi)-definiteness, i.e., whether each $A\in\imace{A}$ is positive (semi)-definite; one has to check $\unum{\lambda}_n>0$ or $\unum{\lambda}_n\geq0$, respectively
\item
Schur or Hurwitz stability, i.e., whether each $A\in\imace{A}$ is stable; for Schur stability, one has to check $\unum{\lambda}_n>-1$ and $\onum{\lambda}_1<1$, and for Hurwitz stability $\onum{\lambda}_1<0$
\item
spectral radius, i.e., the largest spectral radius over $A\in\imace{A}$; it has the value of $\max\{\onum{\lambda}_1,-\unum{\lambda}_n\}$.
\end{itemize}

\section{Checking sign invariancy}

Recall Theorem 7.9.3 from Parlett \cite{Par1998} stated in an adapted formulation.

\begin{theorem}\label{thmEigvecExtr}
If $A\in\imace{A}$ with $b>0$, then there is no eigenvector $x$ such that $x_1=0$ or $x_n=0$.
\end{theorem}

We can utilize this theorem also for the intermediate entries of eigenvectors.

\begin{proposition}\label{propEigvecZerosIn}
If there is $A\in\imace{A}$ with $b>0$, and $x_i=0$ for some eigenvector $x$, then $x_{i-1},x_{i+1}\not=0$.
\end{proposition}

\begin{proof}
From $x_i=0$ we have that $(x_1,\dots,x_{i-1})^T$ is an eigenvector of the  principal leading submatrix of $\imace{A}$ of size $i-1$, and therefore $x_{i-1}\not=0$. Similarly for $x_{i+1}$.
\end{proof}

The following observation is a basis for the method recognizing sign invariancy.
We will denote by $A_{[i:j]}$ the principal submatrix of $A$ indexed by $i,i+1,\dots,j$.

\begin{proposition}\label{propEigComm}
Suppose that $\uvr{b}>0$. Then the problem is not sign invariant if and only if there is $A\in\imace{A}$ and $i\in\seznam{n}$ such that both matrices $A_{[1:(i-1)]}$ and $A_{[(i+1):n]}$ share a common eigenvalue.
\end{proposition}

\begin{proof}
Since $\uvr{b}>0$, by Proposition~\ref{propSimple} the eigenvalues of all $A\in\imace{A}$ are simple, and therefore the corresponding eigenvectors $x(A)$ can be chosen in such a way that they constitute continuous mappings with respect to $A\in\imace{A}$.
Thus the problem is not sign invariant if and only if there an eigenvector with zero entry.

Let $\lambda$ be the eigenvalue corresponding to an eigenvector $x$. If $x_i=0$, then both matrices $A_{[1:(i-1)]}$ and $A_{[(i+1):n]}$ have a common eigenvalue $\lambda$, and the eigenvectors are $(x_1,\dots,x_{i-1})^T$ and $(x_{i+1},\dots,x_n)^T$, respectively.

On the other hand, let $A_{[1:(i-1)]}$ and $A_{[(i+1):n]}$ have a common eigenvalue $\lambda$ corresponding to eigenvectors $(x_1,\dots,x_{i-1})^T$ and $(x_{i+1},\dots,x_n)^T$, respectively. Then $x_1,x_{i-1},x_{i+1},x_n\not=0$ by Theorem~\ref{thmEigvecExtr}, and therefore $\lambda$ is the eigenvalue of $A$ corresponding to the eigenvector $(x_1,\dots,x_{i-1},0,\alpha x_{i+1},\dots,\alpha x_n)^T$ for some $\alpha\not=0$.
\end{proof}

\paragraph{The method for checking sign invariancy}
For any $I\subseteq\seznam{n}$ do the following.
The index set $I=\{i_1,\dots,i_k\}$ represents zero entries of an eigenvector. Consider the interval principal submatrices $\imace{A}_{[1:(i_1-1)]},\dots,\imace{A}_{[(i_k+1):n]}$ associated with $I$. Compute their inner estimation eigenvalue intervals by Remark~\ref{remInner}. If there is a common value $\lambda$, then the problem is not sign invariant by Proposition~\ref{propEigComm}.

If the test passes successfully through every $I\subseteq\seznam{n}$, then the problem is sign invariant. The reason is the following. Let $x$ be an eigenvector of any $A\in\imace{A}$ with the most zero entries. Let $I$ be the index set of the zero entries. Then the problem becomes sign invariant on principal submatrices $\imace{A}_{[1:(i_1-1)]},\dots,\imace{A}_{[(i_k+1):n]}$, and therefore we must find a common eigenvalue.

Notice that not all the number $2^n$ of index sets $I\subseteq\seznam{n}$ are necessary to process. By Theorem~\ref{thmEigvecExtr} and Proposition~\ref{propEigvecZerosIn} only certain index sets can be considered. What is the number of such sets? Denote it by $p_n$. We easily find a Fibonacci-type recurrence relation $p_n=p_{n-1}+p_{n-2}$ since either $1,3\not\in I$, $2\in I$, or $1,2\not\in I$. Therefore $p_n$ asymptotically grows as $p_n\approx 1.618^n$, which is still exponential, but significantly less than $2^n$.

The following gives a sufficient condition for sign invariancy.
Denote by $\Lambda(\imace{A})$ any superset of the eigenvalue sets of $\imace{A}$, that is,
$$
\cup_{i=1}^n \inum{\lambda}_i\subseteq \Lambda(\imace{A}).
$$
Methods for computing such outer estimations were addressed, e.g., in \cite{HlaDan2010,HlaDan2011b,Kol2010,Len2014,QiuChe1996}.

\begin{proposition}\label{propInvSuffOuter}
The problem is sign invariant if $\uvr{b}>0$ and $\Lambda(\imace{A}_{[1:(i-1)]}) \cap \Lambda(\imace{A}_{[(i+1):n]})=\emptyset$ for every $k=2,\dots,n-1$.
\end{proposition}

\begin{proof}
It follows from Proposition~\ref{propEigComm}.
\end{proof}

\section{Special case of disjoint eigenvalue sets}

An interval matrix $\imace{A}$ is called regular if every $A\in\imace{A}$ is nonsingular; see \cite{Roh2009}. Bar-On et al.\ \cite{Bar2000,BarCod1996} showed that checking regularity of a tridiagonal interval matrix is a polynomial problem. Their algorithm works analogously even if we restrict to symmetric matrices in $\imace{A}$. Therefore, checking regularity of a tridiagonal interval matrix can be checked efficiently. As a direct consequence, checking whether a given $\lambda\in\R$ is an eigenvalue of at least one $A\in\imace{A}$ is a polynomially solvable problem, too.

We use this observation for computing the corresponding eigenvalue sets in the case the eigenvalue sets $\inum{\lambda}_1,\dots,\inum{\lambda}_n$ are mutually disjoint. Compute the inner estimation $\inum{\mu}_1,\dots,\inum{\mu}_n$ of the eigenvalue sets by Remark~\ref{remInner}. If these intervals are mutually disjoint, then for each $i=2,\dots,n-1$ check by the above observation whether $\unum{\mu}_i-\eps$ or $\onum{\mu}_i+\eps$ belong to the eigenvalue set for a sufficiently small $\eps>0$. (To avoid numerical difficulties, one can consider $\eps$ as a parameter.) If it is not the case, then empty pairwise intersection of the eigenvalue sets is confirmed. Eventually, we have $\inum{\lambda}_i=\inum{\mu}_i$ for every $i=1,\dots,n$.

\section{Examples}

\begin{example}
Consider the example from \cite{HlaDan2010,HlaDan2011b,Len2014,QiuChe1996}:
\begin{align*}
\imace{A}=\begin{pmatrix}
[2975,3025]& [-2015,-1985]& 0&0\\
[-2015,-1985]&[4965,5035]& [-3020,-2980]& 0\\
0& [-3020,-2980] &[6955,7045]& [-4025,-3975]\\
0 &0& [-4025,-3975]& [8945,9055]\end{pmatrix}.
\end{align*}
First, we transform the matrix into a nonnegative one
\begin{align*}
\imace{A}'=\begin{pmatrix}
[2975,3025]& [1985,2015]& 0&0\\
[1985,2015]&[4965,5035]& [2980,3020]& 0\\
0& [2980,3020] &[6955,7045]& [3975,4025]\\
0 &0& [3975,4025]& [8945,9055]\end{pmatrix}.
\end{align*}
The eigenvalues of the midpoint matrix are $\lambda_1=12641$, $\lambda_2=7064.5$, $\lambda_3=3389.9$, $\lambda_4=905.17$, and the corresponding eigenvectors are
\begin{align*}
v_1&=(0.05575,0.26874,0.64725,0.71116)^T,\  &
v_2&=(-0.3546,-0.7206,-0.2595,0.5363)^T,\\
v_3&=(0.71884,0.14012,-0.55442,0.39531)^T,\  &
v_4&=(0.59535,-0.62357,0.45425,-0.22446)^T.
\end{align*}
Based on the signs of the entries of these vectors we can directly conclude that $\onum{\lambda}_1$ is attained for $\omace{A}$, and similarly $\onum{\lambda}_2$, $\onum{\lambda}_3$, $\onum{\lambda}_4$ are attained as the corresponding eigenvalues of the matrices
\begin{align*}
\begin{pmatrix}
 3025 & 2015 &    0 & 0\\
 2015 & 5035 & 3020 & 0\\
    0 & 3020 & 7045 & 3975\\
    0 &    0 & 3975 & 9055
\end{pmatrix},\ 
\begin{pmatrix}
 3025 & 2015 &    0 & 0\\
 2015 & 5035 & 2980 & 0\\
    0 & 2980 & 7045 & 3975\\
    0 &    0 & 3975 & 9055
\end{pmatrix},\ 
\begin{pmatrix}
 3025 & 1985 &    0 & 0\\
 1985 & 5035 & 2980 & 0\\
    0 & 2980 & 7045 & 3975\\
    0 &    0 & 3975 & 9055
\end{pmatrix},
\end{align*}
respectively. Similarly we proceed for calculating the lower end-points of the eigenvalue sets. Eventually, we obtain the following exact eigenvalue sets (by using outward rounding)
\begin{align*}
\inum{\lambda}_1&=[12560.8377,\,12720.2273],&\ \ 
\inum{\lambda}_2&=[7002.2827,\,7126.8283],\\ 
\inum{\lambda}_3&=[3337.0784,\,3443.3128],&\ \ 
\inum{\lambda}_4&=[842.9250,\,967.1083].
\end{align*}
\end{example}

\section{Conclusion}

We presented a simple and fast algorithm for computing the eigenvalue ranges of symmetric tridiagonal interval matrices. Impreciseness of measurement and other kinds of uncertainty are often represented in the form of intervals. Therefore, checking various kinds of stability of uncertain systems naturally leads to the problem of determining eigenvalues of interval matrices. In this short note, we improved the time complexity and the overall exposition of the known methods for the symmetric tridiagonal matrix case.

\section*{Acknowledgments} 

The author was supported by the Czech Science Foundation Grant P402/13-10660S.


\bibliographystyle{abbrv}
\bibliography{tridiag}

\begin{thebibliography}{10}

\bibitem{AhnMoor2006}
H.-S. Ahn, K.~L. Moore, and Y.~Chen.
\newblock Monotonic convergent iterative learning controller design based on
  interval model conversion.
\newblock {\em IEEE Trans. Autom. Control}, 51(2):366--371, 2006.

\bibitem{Bar2000}
I.~Bar-On.
\newblock Checking nonsingularity of tridiagonal matrices.
\newblock {\em Electron. J. Linear Algebra}, 6:11--19, 2000.

\bibitem{BarCod1996}
I.~Bar-On, B.~Codenotti, and M.~Leoncini.
\newblock Checking robust nonsingularity of tridiagonal matrices in linear
  time.
\newblock {\em BIT}, 36(2):206--220, 1996.

\bibitem{Com1994}
J.~C. Commer\c{c}on.
\newblock Eigenvalues of tridiagonal symmetric interval matrices.
\newblock {\em IEEE Trans. Autom. Control}, 39(2):377--379, 1994.

\bibitem{DeiRoh1994}
A.~Deif and J.~Rohn.
\newblock On the invariance of the sign pattern of matrix eigenvectors under
  perturbation.
\newblock {\em Linear Algebra Appl.}, 196:63--70, 1994.

\bibitem{Dei1991a}
A.~S. Deif.
\newblock The interval eigenvalue problem.
\newblock {\em ZAMM, Z. Angew. Math. Mech.}, 71(1):61--64, 1991.

\bibitem{Her1992}
D.~Hertz.
\newblock The extreme eigenvalues and stability of real symmetric interval
  matrices.
\newblock {\em IEEE Trans. Autom. Control}, 37(4):532--535, 1992.

\bibitem{Hla2015b}
M.~Hlad\'{\i}k.
\newblock Complexity issues for the symmetric interval eigenvalue problem.
\newblock {\em Open Math.}, 13(1):157--164, 2015.

\bibitem{HlaDan2010}
M.~Hlad\'{\i}k, D.~Daney, and E.~Tsigaridas.
\newblock Bounds on real eigenvalues and singular values of interval matrices.
\newblock {\em SIAM J. Matrix Anal. Appl.}, 31(4):2116--2129, 2010.

\bibitem{HlaDan2011c}
M.~Hlad\'{\i}k, D.~Daney, and E.~P. Tsigaridas.
\newblock Characterizing and approximating eigenvalue sets of symmetric
  interval matrices.
\newblock {\em Comput. Math. Appl.}, 62(8):3152--3163, 2011.

\bibitem{HlaDan2011b}
M.~Hlad\'{\i}k, D.~Daney, and E.~P. Tsigaridas.
\newblock A filtering method for the interval eigenvalue problem.
\newblock {\em Appl. Math. Comput.}, 217(12):5236--5242, 2011.

\bibitem{Jia2017}
Y.~Jian.
\newblock Extremal eigenvalue intervals of symmetric tridiagonal interval
  matrices.
\newblock {\em Numer. Linear Algebra Appl.}, 24(2):e2083, 2017.

\bibitem{Kol2007}
L.~V. Kolev.
\newblock Determining the positive definiteness margin of interval matrices.
\newblock {\em Reliab. Comput.}, 13(6):445--466, 2007.

\bibitem{Kol2010}
L.~V. Kolev.
\newblock Eigenvalue range determination for interval and parametric matrices.
\newblock {\em Int. J. Circuit Theory Appl.}, 38(10):1027--1061, 2010.

\bibitem{Len2014}
H.~Leng.
\newblock Real eigenvalue bounds of standard and generalized real interval
  eigenvalue problems.
\newblock {\em Appl. Math. Comput.}, 232:164--171, 2014.

\bibitem{Par1998}
B.~N. {Parlett}.
\newblock {\em The symmetric eigenvalue problem}.
\newblock SIAM, Philadelphia, unabridged, corrected republication of 1980
  edition, 1998.

\bibitem{QiuChe1996}
Z.~Qiu, S.~Chen, and I.~Elishakoff.
\newblock Bounds of eigenvalues for structures with an interval description of
  uncertain-but-non-random parameters.
\newblock {\em Chaos Soliton. Fract.}, 7(3):425--434, 1996.

\bibitem{Roh1993b}
J.~Rohn.
\newblock Interval matrices: {Singularity} and real eigenvalues.
\newblock {\em SIAM J. Matrix Anal. Appl.}, 14(1):82--91, 1993.

\bibitem{Roh1994}
J.~Rohn.
\newblock Checking positive definiteness or stability of symmetric interval
  matrices is {NP}-hard.
\newblock {\em Commentat. Math. Univ. Carol.}, 35(4):795­--797, 1994.

\bibitem{Roh1996c}
J.~Rohn.
\newblock An algorithm for checking stability of symmetric interval matrices.
\newblock {\em IEEE Trans. Autom. Control}, 41(1):133--136, 1996.

\bibitem{Roh2009}
J.~Rohn.
\newblock Forty necessary and sufficient conditions for regularity of interval
  matrices: {A} survey.
\newblock {\em Electron. J. Linear Algebra}, 18:500--512, 2009.

\end{thebibliography}

\end{document}